\documentclass{amsart}
\usepackage{amssymb, amsmath, latexsym, mathabx, bm, dsfont}
\usepackage{xcolor}
\usepackage{stmaryrd}
\usepackage{makecell}

\renewcommand{\baselinestretch}{\baselinestretch}
\renewcommand{\baselinestretch}{1.1}
\numberwithin{equation}{section}

\newtheorem{thm}{Theorem}[section]
\newtheorem{lem}[thm]{Lemma}

\newtheorem{prop}[thm]{Proposition}

\newtheorem{rmk}[thm]{Remark}

\newcommand{\lra}{{\ \longrightarrow \ }}
\newcommand{\nlra}{{\ \longarrownot\longrightarrow \ }}

\newcommand{\gen}{\text{gen}}

\newcommand{\rank}{\text{rank}}

\newcommand{\n}{{\mathbb N}}
\newcommand{\z}{{\mathbb Z}}
\newcommand{\q}{{\mathbb Q}}

\newcommand{\Mod}[1]{\ (\mathrm{mod}\ #1)}

\newcommand{\bx}{\bm x}
\newcommand{\by}{\bm y}

\newcommand{\fs}{\mathfrak s}

\begin{document}

\title{Tight universal quadratic forms}

\author{Mingyu Kim}
\address{Department of Mathematics, Sungkyunkwan University, Suwon 16419, Korea}
\email{kmg2562@skku.edu}
\thanks{This research of the first author was supported by Basic Science Research Program through the National Research Foundation of Korea(NRF) funded by the Ministry of Education (NRF-2019R1A6A3A01096245) and (NRF-2021R1C1C2010133).}

\author{Byeong-Kweon Oh}
\address{Department of Mathematical Sciences and Research Institute of Mathematics, Seoul National University, Seoul 151-747, Korea}
\email{bkoh@snu.ac.kr}
\thanks{This work of the second author was supported by the National Research Foundation of Korea (NRF-2019R1A2C1086347)  and (NRF-2020R1A5A1016126).}

\subjclass[2010]{Primary 11E12, 11E20, 11E41}

\keywords{Tight universal quadratic forms}

\begin{abstract}  For a positive integer $n$, let $\mathcal T(n)$ be the set of all integers greater than or equal to $n$. An integral quadratic form $f$ is called tight $\mathcal T(n)$-universal if the set of nonzero integers that are represented by $f$ is exactly $\mathcal T(n)$. The smallest possible rank over all  tight $\mathcal T(n)$-universal quadratic forms is defined by $t(n)$.     
In this article, we find all tight $\mathcal T(n)$-universal diagonal quadratic forms.  We also prove that $t(n) \in \Omega(\log_2(n)) \cap O(\sqrt{n})$. Explicit lower and upper bounds for $t(n)$ will be provided for some small integer $n$.    
\end{abstract} 

\maketitle

\section{Introduction}
For a positive integer $k$, an integral quadratic form $f$ of rank $k$ is a homogeneous quadratic polynomial 
$$
f(x_1,x_2,\dots,x_k)=\sum_{i,j=1}^k f_{ij}x_ix_j \quad (f_{ij}=f_{ji} \in \z)
$$
with $k$ variables such that the discriminant $\det(f_{ij})$  is nonzero.  We say an integer $a$ is represented by $f$ if there is an integer solution of the diophantine equation $f(x_1,x_2,\dots,x_k)=a$. The set of ``nonzero" integers that are represented by $f$ is denoted by $Q(f)$. It is quite an old problem in number theory to determine the set $Q(f)$ explicitly for any given quadratic form $f$. If $f$ is indefinite, that is, the corresponding symmetric matrix $(f_{ij})$ is indefinite,  and if the rank of $f$ is greater than $2$, then so called, the spinor genus theory works to determine the set $Q(f)$, effectively (for this, see \cite{hs}).  From now on, we always assume that $f$ is positive definite.  If $k$ is greater than $3$, then there is a general method to determine $Q(f)$ under the assumption that the minimum representation number $\min(f)$ is not too big (for this, see  \cite{han} and \cite{hi}). If $k=3$ or $\min(f)$ is  large, then, as far as the authors know,  only  sufficiently large integers in $Q(f)$  can effectively be determined. 

Let $n$ be a positive integer, and let $\mathcal T(n)$ be the set of all integers greater than or equal to $n$. We say an integral quadratic form $f$ is {\it $\mathcal T(n)$-universal} if $f$ represents all integers greater than or equal to $n$, that is, $\mathcal T(n) \subseteq Q(f)$.  We say a $\mathcal T(n)$-universal quadratic form $f$ is {\it tight} if $\mathcal T(n)=Q(f)$. A tight $\mathcal T(1)$-universal quadratic form is simply called  {\it universal}.   The smallest possible rank over all  tight $\mathcal T(n)$-universal quadratic forms is defined by $t(n)$.  The famous Lagrange's four square theorem  says that $t(1)=4$, that is, any positive integer is a sum of four squares, but no less.  Note that the sum of four squares $x^2+y^2+z^2+t^2$ is $\mathcal T(n)$-universal for any $n\ge 1$, whereas it is not  tight  $\mathcal T(n)$-universal for any $n\ge 2$.  
One may easily show that there does not exist a ternary  $\mathcal T(n)$-universal quadratic form for any positive integer $n$.  Conway and Schneeberger proved that there are exactly $204$ quaternary universal quadratic forms up to isometry (for this, see \cite {b}).  Halmos proved  in \cite{hal}  that the quaternary quadratic form $2x^2+2y^2+3z^2+4t^2$ represents all integers greater than $1$. Hence it is a tight $\mathcal T(2)$-universal quadratic form, and therefore we have $t(2)=4$. Barowsky and his/her collaborators proved  in \cite{bar} that the quaternary quadratic form 
$$
3x^2-2xy-2xz+2xt+3y^2-2yt+4z^2-4zt+5t^2
$$
 represents all integers except for $1$ and $2$. Hence it is tight $\mathcal T(3)$-universal and therefore $t(3)=4$.  As far as the authors know, there is no known tight $\mathcal T(n)$-universal quadratic form with rank $t(n)$ for any integer $n\ge 4$. 

In this article, we find all tight $\mathcal T(n)$-universal quadratic forms which are ``diagonal".  This will be done in Section 2. In Section 3, We  provide an explicit upper bound for $t(n)$ for any integer $n \le 36$. In Section 4, we prove that  $t(n) \in \Omega(\log_2(n)) \cap O(\sqrt{n})$. In particular, the minimum rank of tight $\mathcal T(n)$-universal quadratic forms goes to infinity as $n$ increases.  Recall that for two arithmetic functions $f(n)$ and $g(n)$, we say $f(n) \in \Omega(g(n))$ if and only if there is a constant $C>0$ such that $C\cdot g(n) \le f(n)$ for any sufficiently large integer $n$.  

The subsequent  discussion will be conducted in the language of quadratic spaces and lattices.  The readers are referred to \cite{ki} and \cite{om} for any unexplained notations and terminologies.    For simplicity, the quadratic map and its associated bilinear form on any quadratic space will be denoted by $Q$ and $B$, respectively.  The term {\em lattice} always means a finitely generated $\z$-module on a finite dimensional positive definite quadratic space over $\q$.  The scale of a lattice $L$, denoted $\mathfrak s(L)$, is the ideal generated by $\{B(\bx, \by): \bx, \by \in L\}$ in $\z$.  We call $L$ an integral lattice if $\fs(L) \subseteq \z$.  Throughout this article, we always assume that a $\z$-lattice is {\it positive definite and integral}.  

Let $L=\z \bx_1+\z \bx_2+\dots+\z \bx_n$ be a $\z$-lattice of rank $n$. Note that the quadratic form $f_L$ corresponding to $L$ is defined by $f_L(x_1,x_2,\dots,x_n)=\sum_{i,j=1}^n B(\bx_i,\bx_j) x_ix_j$. The symmetric matrix $(B(\bx_i,\bx_j))$ is called the Gram matrix of the lattice $L$. 
 Let $p$ be a prime and let $\z_p$ be the $p$-adic integer ring. We define $L_p=L\otimes \z_p$, which is considered as  a $\z_p$-lattice. 

 A $\z$-lattice $M$ is said to be represented by  $L$ if there is a linear map $\sigma: M \longrightarrow L$ such that $Q(\sigma(\bx)) = Q(\bx)$ for any $\bx \in M$.  Such a map is called a representation from $M$ into $L$, which is necessarily injective because the bilinear map defined on $M$ is assumed to be nondegenerate.  If there is a linear map $\sigma_p : M_p \longrightarrow L_p$ satisfying the above property for some prime $p$, then we say $M$ is represented by $L$ over $\z_p$. We say $M$ is locally represented by $L$ if $M$ is represented by $L$ over $\z_p$ for any prime $p$.   If $M$ is represented by $L$, then we simply write $M \lra L$. In particular, if $M=\z \bx$ with $Q(\bx)=m$ is a unary $\z$-lattice, then we simply write $m \lra L$ as well as $M \lra L$.  We also define 
 $$
Q(L)=\{ a \in \z^{+} : a \lra L\}.
$$ 
 We say $M$ is isometric to $L$ over $\z$  if there exists a representation sending $L$ onto $M$.  In this case we will write $L \cong M$.  If $M$ is isometric to $L$ over $\z_p$ for any prime $p$, then we say $M$ is contained in the genus of $L$, and we write $M \in \gen(L)$.  The number of isometry classes in the  genus of $L$ is called the class number of $L$, and is denoted by $h(L)$. It is well known that the class number of any $\z$-lattice is always finite. It is also well known that a $\z$-lattice $K$ is locally represented by $L$ if and only if there is a $\z$-lattice $M \in \gen(L)$ such that $K \lra M$.   The set of all integers that are locally represented by $L$ is denoted by $Q(\gen(L))$. 
 
If $L$ is a $\z$-lattice and $A$ is one of its Gram matrix, we will write $L \cong A$.   We will often address a positive definite symmetric matrix as a lattice.  The diagonal matrix with entries $a_1, \ldots, a_n$ on its main diagonal will be denoted by $\langle a_1, \ldots, a_n\rangle$.  
 
For two vectors $(u_1,u_2,\dots,u_r)\in \z^r$ and $(v_1,v_2,\dots,v_s)\in \z^s$, we write
$$
(u_1,u_2,\dots,u_r)\preceq (v_1,v_2,\dots,v_s)
$$
if $\{u_i\}_{1\le i\le r}$ is a subsequence of $\{v_j\}_{1\le j\le s}$
and we write
$$
(u_1,u_2,\dots,u_r)\prec (v_1,v_2,\dots,v_s)
$$
if $\{u_i\}_{1\le i\le r}$ is a proper subsequence of $\{v_j\}_{1\le j\le s}$.
For two diagonal $\z$-lattices $\ell=\langle u_1,u_2,\dots,u_r \rangle$ and $L=\langle v_1,v_2,\dots,v_s \rangle$, we define $\ell \preceq L$    ($\ell \prec L$) if and only if 
$(u_1,u_2,\dots,u_r)\preceq  (v_1,v_2,\dots,v_s)$ ($(u_1,u_2,\dots,u_r)\prec  (v_1,v_2,\dots,v_s)$, respectively).

\section{Diagonal tight $\mathcal T(n)$-universal $\z$-lattices}   

In this section, we find all diagonal tight $\mathcal T(n)$-universal $\z$-lattices for any positive integer $n$.   Let $L=\langle a_1,a_2,\dots,a_k\rangle$ be a diagonal tight $\mathcal T(n)$-universal $\z$-lattice.
We say that $L$ is \textit{new} if $\langle a_{j_1},a_{j_2},\dots,a_{j_l}\rangle$ is not $\mathcal T(n)$-universal whenever $\langle a_{j_1},a_{j_2},\dots,a_{j_l}\rangle \prec \langle a_1,a_2,\dots,a_k\rangle$.  From the definition,  to find all diagonal tight $\mathcal T(n)$-universal $\z$-lattices, it suffices to find all ``new''    diagonal tight $\mathcal T(n)$-universal $\z$-lattices.  Ramanujan proved in \cite{r} that there are exactly $54$ diagonal tight $\mathcal T(1)$-universal quaternary $\z$-lattices, which are all new (see also \cite{dic}). One may easily show by using, so called,  the $15$-theorem (see \cite{b}) that all new diagonal tight $\mathcal T(1)$-universal $\z$-lattices with rank greater than $4$ are
$$
\langle 1,2,5,5,a\rangle, \quad \text{where $a=5,11,12,13,14$, and $15$.}
$$   

From now on, we always assume that $n\ge2$.

\begin{lem} \label{lemdiag1}
Let $L=\langle a_1,a_2,\dots,a_k\rangle$ be a diagonal $\z$-lattice  with $2\le a_1\le a_2\le \cdots \le a_k$.
Assume that there are positive integers $m$ and $i$ such that the following conditions hold;
\begin{enumerate}
\item $m\lra L$;
\item $m\nlra \langle a_1,\dots,a_i\rangle$;
\item $m<a_i+a_1$.
\end{enumerate}
Then there is an integer $j$ with $i+1\le j\le k$ such that $a_j=m$.
\end{lem}

\begin{proof}
Note that there is a vector $(x_1,x_2,\dots,x_k)\in \z^k$ such that
$$
a_1x_1^2+a_2x_2^2+\cdots+a_kx_k^2=m
$$
by condition (1). Furthermore, by condition (2),  there is an index $j$ with $i+1\le j\le k$ such that $x_j\neq 0$.
Now, by condition (3), we have
$$
m<a_i+a_1\le a_j+a_1,
$$
and thus $x_j=\pm 1$ and $x_s=0$ for any $s\ne j$. 
This completes the proof.
\end{proof}

\begin{lem} \label{lemdiag2}  
Let $L=\langle a_1,a_2,\dots,a_k\rangle$ be a diagonal $\mathcal T(a_1)$-universal $\z$-lattice with $2 \le a_1\le a_2\le \dots \le a_k$.
Let $r$ and $s$ be positive integers with $r<s$ such that $r\le a_1-1$ and $s\le k-1$.
Assume further  that there are integers $a_s<m_1<m_2<\dots <m_r<a_s+a_1$ and $1\le l_1<l_2<\cdots<l_{s-r}\le s$ such that the following conditions hold;
\begin{enumerate}
\item $m_i\nlra \langle a_1,a_2,\dots,a_s\rangle$ for any $i=1,2,\dots,r$;
\item $\langle a_{l_1},a_{l_2},\dots,a_{l_{s-r}},m_1,m_2,\dots, m_r\rangle$ is $\mathcal T(a_1)$-universal.
\end{enumerate}
Then $L$ is not new.
\end{lem}

\begin{proof}
For each $\nu =1,2,\dots,r$, there is an integer $j_{\nu}$ with $s+1\le j_{\nu}\le k$ such that $a_{j_{\nu}}=m_{\nu}$ by Lemma \ref{lemdiag1}. Thus it follows from the condition (2) that $L$ is not new.
\end{proof}

For any integer $n$ with $n\ge 2$, we define two diagonal $\z$-lattices
$$
X_n:=\langle n,n+1,n+2,\dots,2n\rangle \quad \text{and} \quad 
Y_n:=\langle n,n,n+1,n+2,\dots,2n-1\rangle.
$$
\begin{prop} \label{propxy}
Let $n$ be an integer greater than $1$ and let $L$ be a diagonal tight $\mathcal T(n)$-universal $\z$-lattice.
Then we have $X_n\preceq L$ or $Y_n\preceq L$.
\end{prop}

\begin{proof} Assume that $L=\langle a_1,a_2,\dots,a_k\rangle$, where $a_1\le a_2 \le\dots\le a_k$. 
Clearly $a_1=n$. Now, we apply Lemma \ref{lemdiag1}  on the case when $i=\nu$ and $m=n+\nu$ for each $\nu=1,2,\dots,n-1$ to prove that
$$
\langle n,n+1,n+2,\dots,2n-1\rangle \prec L.
$$
Since $2n\lra L$ and  $2n\nlra \langle n,n+1,\dots,2n-1\rangle$, one may easily show that
$$
X_n=\langle n,n+1,n+2,\dots,2n-1,2n\rangle \preceq L\ \ \text{or}\ \ Y_n=\langle n,n,n+1,n+2,\dots,2n-1\rangle \preceq L.
$$
This completes the proof.
\end{proof}

\begin{prop} \label{propdiag1}
For any positive integer $n$ greater than $3$,
the diagonal $\z$-lattice $X_n=\langle n, n+1, \dots, 2n \rangle$ is tight $\mathcal T(n)$-universal. In particular, we have $t(n) \le n+1$ for any $n\ge 4$.
\end{prop}

\begin{proof}
Note that any positive integer $k$ greater than or equal to $n$ is of the form $nu+a$ for some nonnegative  integer $u$ and some integer $a$ with  $n \le a \le 2n-1$.

First, assume that $n$ is greater than or equal to 7.
Let $a$ be any integer with $n \le a \le 2n-1$.
One may choose two integers $n_{1}$ and $n_{2}$ with $0<n_{1}<n_{2} \le \left[ \frac{n-1}{2}\right]$ such that  $\{n_1,n_2\} \cap  \{a-n, 2n-a\}=\emptyset$.
Since $\langle 1,2,3,3 \rangle$ is universal, the diagonal $\z$-lattice 
$$
\langle n,2n,n+n_{1}, 2n-n_{1},n+n_{2},2n-n_{2} \rangle
$$
 represents all nonnegative integers which are multiples of $n$. Furthermore,  the diagonal $\z$-lattice $\langle n,2n,n+n_{1},2n-n_{1},n+n_{2},2n-n_{2},a\rangle$ represents all positive integers of the form $nu+a$, where $n+1 \le a\le 2n-1$. Therefore $\langle n, n+1, \dots, 2n \rangle$ represents all integers greater than or equal to $n$.

Secondly, assume that $n=6$. Note that by the $15$-Theorem, the quaternary $\z$-lattices $\langle 1,2,3,3\rangle$ and $\langle 1,2,3,4\rangle$ are universal.  
Since $\langle 6, 12, 18, 36 \rangle$ is a $\z$-sublattice of $\langle 6,6+n',9,12-n',12 \rangle$ for any $n'=1,2$, the diagonal $\z$-lattice $\langle 6,7,8,9,10,11,12 \rangle$ represents all integers of the form $6u+a$ with $a=7,8,10,11$.
On the other hand, $\langle 6,12,18,18 \rangle$ is a $\z$-sublattice of $\langle 6,7,8,10,11,12 \rangle$.
Therefore $\langle 6,7,8,9,10,11,12 \rangle$ represents all integers of the form $6u+6$ or $6u+9$ for any integer $u$ with $u\ge 0$.

Now, we consider the case when $n=5$. Since all the other cases can be proved in a similar manner, we only prove that any integer of the form $5k+6$ for some nonnegative integer $k$ is represented by $\langle 5,6,7,8,9,10 \rangle$.  
To prove this, we use the well-known fact that $\langle 1,2,3 \rangle$ represents all nonnegative integers except for integers of the form $4^{r}(16s+10)$ with some nonnegative integers $r$ and $s$.
Note that $\langle 5,10,15 \rangle$ is a $\z$-sublattice of  $\langle 5,7,8,10 \rangle$.
For any integer $k$ with $k \ge 15$,
since either $k$ or $k-15$ is not of the form $4^{r}(16s+10)$,  $5k+6=5(k-15)+9 \cdot 3^{2}$ is represented by $\langle 5,6,7,8,9,10 \rangle$.
One may directly check that $\langle 5,6,7,8,9,10 \rangle$ also represents all integers of the form $5k+6$ for any integer $k$ with  $0 \le k \le 14$. 

Finally, assume that $n=4$.
Note that $\langle 1,2,5,6 \rangle$ and $\langle 1,2,5,7 \rangle$ are, in fact, universal.
Since $\langle 4,8,20,24 \rangle$ is a $\z$-sublattice of $\langle 4,5,6,8 \rangle$ and
$\langle 4,8,20,28 \rangle$ is a $\z$-sublattice of $\langle 4,5,7,8 \rangle$, the $\z$-lattice
$\langle 4,5,6,7,8 \rangle$ represents all integers of the form $4u+4$, $4u+6$ and $4u+7$. Hence it suffices to show that any integer of the form $4u+5$ is represented by
 $\langle 4,5,6,7,8 \rangle$.  To prove this, we use the fact that that $\langle 1,2,6 \rangle$ represents all nonnegative integers except for integers of the form $4^{r}(8s+5)$ with some nonnegative integers $r$ and $s$. Since $\langle 4,8,24 \rangle$ is a sublattice of $\langle 4,6,8\rangle$, the ternary $\z$-lattice $\langle 4,6,8\rangle$ represents all integers of the form $4u$, where $u$ is an integer which is represented by $\langle 1,2,6\rangle$.  
For any integer $k$ with $k \ge 10$,   since at least one of  $k$, $k-7$, or $k-10$ is not of the form $4^{r}(8s+5)$, the integer 
$$
4k+5=4(k-7)+7 \cdot 2^{2}+5=4(k-10)+5 \cdot 3^{2}
$$
is represented by $\langle 4,5,6,7,8 \rangle$. Finally, one may directly check that $\langle 4,5,6,7,8 \rangle$ also represents all integers of the form $4k+5$ for any integer $k$ with $0 \le k \le 9$. This completes the proof.  
\end{proof}

\begin{rmk}
{\rm The ternary diagonal $\z$-lattice  $\langle 2,3,4\rangle$ does not represent $10$, whereas the quaternary diagonal $\z$-lattice $\langle 2,3,4,5\rangle$ is, in fact, tight $\mathcal T(2)$-universal. The quaternary diagonal $\z$-lattice $\langle 3,4,5,6\rangle$ does not represent $35$, whereas  $\langle 3,4,5,6,7 \rangle$  is tight $\mathcal T(3)$-universal. These will be proved in Proposition \ref{propg3} and Proposition \ref{propg2}, respectively.}
\end{rmk}
\begin{prop} \label{propdiag2}
For any positive integer $n$ greater than 3, the diagonal $\z$-lattice $Y_n=\langle n,n,n+1,n+2,\dots,2n-1\rangle$ is tight $\mathcal T(n)$-universal.
\end{prop}

\begin{proof}
Note that the quaternary $\z$-lattice  $\langle 1,1,3,3\rangle$ is universal and the ternary $\z$-lattice $\langle 1,1,3\rangle$ represents all positive integers except for those integers which are  of the form $3^{2a+1}(3b+2)$ for some nonnegative integers $a$ and $b$.
If $n\ge 5$,  the proof of the tight $\mathcal T(n)$-universality of $Y_n$  is quite similar to that of Proposition \ref{propdiag1}. So, the proof is left as an exercise to the reader. 
Now, we show that $Y_4=\langle 4,4,5,6,7\rangle$ is tight $\mathcal T(4)$-universal.
Let $u$ be an integer greater than or equal to $4$. 
One may directly check that $u\lra Y_4$  for any integer $u$  less than or equal to $112$.
Hence we always assume that $u\ge 113$. 
Let $K=\langle 4,4,6\rangle$.  Since $h(K)=1$,  one may easily check by using \cite {om2} that every integer in the set
$$
A=\{ u\in \n : u\equiv 0\Mod 2,\ u\not\equiv 2\Mod {16},\ u\not\equiv 0\Mod 3\}
$$
is represented by $K$. One may easily check that there is a pair 
$$
(d_1,d_2) \in \{ (0,1),(0,2),(0,3),(0,4),(1,0),(2,0),(3,0),(4,0)\}
$$
such that $u-5d_1^2-7d_2^2\in A$.
It follows that $u-5d_1^2-7d_2^2\lra K$ and thus we have $u\lra Y_4$.  This completes the proof.
\end{proof}
Now, by Propositions \ref{propxy}, \ref{propdiag1}, and \ref{propdiag2}, we have
\begin{thm} \label{thmge4}
For any positive integer $n$ greater than 3, there are exactly two new diagonal tight $\mathcal T(n)$-universal $\z$-lattices, which are, in fact, $X_n$ and $Y_n$.
\end{thm}

\begin{table}[ht]
\caption{New diagonal tight $\mathcal T(3)$-universal $\z$-lattices $\langle a_1,a_2,\dots,a_k\rangle$} 
\label{tableg3}
	\begin{center}
		\begin{tabular}{|cccccc|c|}
			\Xhline{1pt}
			$a_1$&$a_2$&$a_3$&$a_4$&$a_5$&$a_6$& Conditions on $a_k  \ (5\le k\le6)$ \\
			\hline
			3&3&3&4&5&&\\
			3&3&4&4&4&5&\\			
			3&3&4&4&5&$a_6$&$5\le a_6\le 18,\ a_6\neq 6,8,9,10,16,17$\\
			3&3&4&5&5&$a_6$&$5\le a_6\le 18,\ a_6\neq 6,8,9,10,16,17$\\
			3&3&4&5&$a_5$&&$6\le a_5\le 10,\ a_5\neq 7$\\
			3&3&4&5&7&$a_6$&$7\le a_6\le 25,\ a_6\neq 8,9,10,23,24$\\
			3&3&4&5&11&13&\\
			3&3&4&5&13&14&\\
			\hline
			3&4&4&5&6&&\\
			3&4&5&5&6&&\\	
			3&4&5&6&$a_5$&&$6\le a_5\le 35,\ a_5\neq 11,33,34$\\
			3&4&5&6&11&$a_6$&$a_6=11$ or $33\le a_6\le 46,\ a_6\neq 35,44,45$\\
			\Xhline{1pt}
		\end{tabular}
	\end{center}
\end{table}

\begin{prop} \label{propg3}
Every diagonal $\z$-lattice listed in Table \ref{tableg3} is tight $\mathcal T(3)$-universal.
\end{prop}

\begin{proof}
Let $L$ be any $\z$-lattice in Table \ref{tableg3}. One may directly check that $L$ represents all integers $m$ with $3\le m\le 2205$.  Furthermore, since $X_3$ or $Y_3$ is represented by $L$,  it suffices to show that both $X_3$ and $Y_3$ represent all integers greater than $2205$.
Let $u$ be an integer greater than $2205$.

First, assume that $X_3\lra L$.
Let $K_1=\langle 3,4,6\rangle$. Note that $h(K_1)=2$ and the other $\z$-lattice in the genus of $K_1$ is $K_2=\langle 1,6,12\rangle$.
If an even positive integer $w$ is represented by $K_2$, then $w\lra \langle 4,6,12\rangle$ and thus $w\lra K_1$. Hence any even integer that is locally represented by  $K_1$ is represented by $K_1$ globally.  
Now, one may easily verify that there is an integer 
$$
d\in \{0,1,2,3,4,5,6,7,9,11,15, 21\}
$$
such that $u-5d^2$ is an even positive integer that is locally represented by $K_1$.
It follows that $u-5d^2\lra K_1$ and thus we have $u\lra X_3$.

Now, assume that $Y_3 \lra L$.
Let $K=\langle 3,3,4\rangle$.
Note that $h(K)=1$.
One may easily verify that there is an integer $d\in \{0,1,2,3,4,5\}$ such that $u-5d^2$ is locally represented by $K$.
Thus $u-5d^2\lra K$ and we have $u\lra Y_3$.
This completes the proof.
\end{proof}

Let $n$ be a positive integer. 
For any positive integers $b_1\le b_2\le \cdots \le b_k$, we define a set $\Psi(b_1,b_2,\dots,b_k)$ by
$$
\Psi(b_1,b_2,\dots,b_k)=\left\{ u\in \mathcal T(n) : u\nlra \langle b_1,b_2,\dots,b_k\rangle \right\}
$$
and we also define $\psi(b_1,b_2,\dots,b_k)$ as
$$
\psi(b_1,b_2,\dots,b_k)=\begin{cases}\infty &\text{if}\ \ \Psi(b_1,b_2,\dots,b_k)=\emptyset,\\
\min\left(\Psi(b_1,b_2,\dots,b_k)\right)&\text{otherwise}.\end{cases}
$$
For $l=1,2,\dots,$ we define sets $A(l)$, $B(l)$, $B'(l)$, and $C(l)$ recursively as follow;
$$
\begin{array}{l}
A(1)=\{(n)\},\\[0.2em]
B(l)=\left\{ (a_1,\dots,a_l)\in A(l) : \psi(a_1,\dots,a_l)=\infty \right\},\\[0.2em]
B'(l)=\left\{ (a_1,\dots,a_l)\in B(l) : (b_1,\dots,b_m)\not\in B(m)\ \text{if}\ (b_1,\dots,b_m)\prec (a_1,\dots,a_l)\right\},\\[0.2em]
C(l)=A(l)- B(l),\\[0.2em]
A(l+1)=\bigcup_{(a_1,\dots,a_l)\in C(l)}\left\{ (a_1,\dots,a_l,a_{l+1})\in \n^{l+1} : a_l\le a_{l+1}\le \psi(a_1,\dots,a_l) \right\}.
\end{array}
$$
From now to the end of this section, we always assume that $a_1\le a_2\le \cdots \le a_k$ whenever we consider a diagonal $\z$-lattice $\langle a_1,a_2,\cdots,a_k\rangle$.
Let $L=\langle a_1,a_2,\dots,a_k\rangle$ be a diagonal tight $\mathcal T(n)$-universal $\z$-lattice. Clearly, $a_1=n$.  From the definition, one may easily see that $L$ is new if and only if $(a_1,a_2,\dots,a_k)\in B'(k)$.
Hence the set of all new diagonal tight $\mathcal T(n)$-universal $\z$-lattices  corresponds exactly to $\bigcup_{k=1}^{\infty} B'(k)$.  
Therefore, to find all ``new'' diagonal tight  $\mathcal T(n)$-universal $\z$-lattices, it suffices to compute the above sets recursively.

\begin{prop} \label{propg3cand}
There are exactly $79$ new diagonal tight $\mathcal T(3)$-universal $\z$-lattices, which are listed in Table \ref{tableg3}.
\end{prop}

\begin{proof} From the above observation, we have to compute the sets $A(l),B(l),B'(l)$, and $C(l)$ on the case when  $n=3$. One may easily compute that
$$
\begin{array}{l}
A(2)=\{ (3,3),(3,4)\},\\[0.3em]
A(3)=\{ (3,3,3),(3,3,4),(3,4,4),(3,4,5)\},\\[0.3em]
A(4)=\left\{ \begin{array}{l}(3,3,3,3),(3,3,3,4),(3,3,4,4),(3,3,4,5),\\
\hskip 1cm (3,4,4,4),(3,4,4,5),(3,4,5,5),(3,4,5,6)\end{array}\right\}.
\end{array}
$$
For $(a_1,a_2,a_3,a_4)\in A(4)$, one may easily check that
$$
\psi(a_1,a_2,a_3,a_4)=\begin{cases} a_4+1&\text{if}\ \ (a_1,a_2,a_3,a_4)\neq (3,3,4,5),(3,4,5,6),\\
13&\text{if}\ \ (a_1,a_2,a_3,a_4)=(3,3,4,5),\\
35&\text{if}\ \ (a_1,a_2,a_3,a_4)=(3,4,5,6).\end{cases}
$$
Hence we have $A(5)=G_1\cup G_2\cup G_3$, where
$$
\begin{array}{l}
G_1=\left\{ \begin{array}{l}(3,3,3,3,3),(3,3,3,3,4),(3,3,3,4,4),(3,3,3,4,5),\\
\hskip 1cm (3,3,4,4,4),(3,3,4,4,5),(3,4,4,4,4),(3,4,4,4,5),\\
\hskip 2cm (3,4,4,5,5),(3,4,4,5,6),(3,4,5,5,5),(3,4,5,5,6)\end{array}\right\},\\[1.6em]
G_2=\left\{ (3,3,4,5,a_5) : 5\le a_5\le 13\right\}, \\[0.3em]
G_3=\left\{ (3,4,5,6,a_5) : 6\le a_5\le 35\right\}.
\end{array}
$$
Among $51$ quinary diagonal $\z$-lattices in $A(5)$, there are exactly $34$ tight $\mathcal T(3)$-universal $\z$-lattices, which are listed in Table \ref{tableg3}.  Note that $\mathcal T(3)$-universalities of those $34$ lattices were proved in Proposition \ref{propg3}.  For each of the remaining $17$ lattices, one may easily find the smallest integer that is not represented by it. From this, one may compute that $\#A(6)=111$.   Among $111$ senary diagonal $\z$-lattices in $A(6)$, there are exactly $85$ tight $\mathcal T(3)$-universal $\z$-lattices. Among those tight $\mathcal T(3)$-universal lattices, only $45$ lattices are new, which are listed in Table \ref{tableg3}.  Summing up all, we have
$$
\begin{array}{llll}
\#A(4)=8,&\#B(4)=0,&\#B'(4)=0,&\#C(4)=8,\\[0.2em]
\#A(5)=51,&\#B(5)=34,&\#B'(5)=34,&\#C(5)=17,\\[0.2em]
\#A(6)=111,&\#B(6)=85,&\#B'(6)=45,&\#C(6)=26.\\[0.2em]
\end{array}
$$

To prove the proposition, it suffices to show that there is no new diagonal tight $\mathcal T(3)$-universal $\z$-lattice of rank greater than 6. Suppose that there is a new diagonal tight $\mathcal T(3)$-universal $\z$-lattice $L=\langle a_1,a_2,\dots,a_k\rangle$ of rank $k\ge 7$.
One may easily deduce $(a_1,a_2,\dots,a_6)\in C(6)$ from the definition of new tight $\mathcal T(3)$-universality.
Note that the set $C(6)$ is as follows:
$$
\left\{ \begin{array}{llll}(3,3,3,3,3,3),&(3,3,3,3,3,4),&(3,3,3,3,4,4),&(3,3,3,4,4,4),\\
(3,3,4,4,4,4),&(3,3,4,4,5,16),&(3,3,4,4,5,17),&(3,3,4,5,5,16),\\
(3,3,4,5,5,17),&(3,3,4,5,7,23),&(3,3,4,5,7,24),&(3,3,4,5,11,11),\\
(3,3,4,5,11,12),&(3,3,4,5,12,12),&(3,3,4,5,12,13),&(3,3,4,5,13,13),\\
(3,4,4,4,4,4),&(3,4,4,4,4,5),&(3,4,4,4,5,5),&(3,4,4,5,5,5),\\
(3,4,5,5,5,5),&(3,4,5,6,11,44),&(3,4,5,6,11,45),&(3,4,5,6,33,33),\\
(3,4,5,6,33,34),&(3,4,5,6,34,34)&&
\end{array}\right\}
$$
Since all the other cases can be proved in a similar manner, we only provide the proof of the first case when $(a_1,a_2,\dots,a_6)=(3,3,3,3,3,3)$.
Suppose that $\ell=\langle 3,3,3,3,3,3,a_7,\dots,a_k\rangle$ is new tight $\mathcal T(3)$-universal for some integers $a_7,\dots,a_k$ with $3\le a_7\le\dots \le a_k$.  
Since $4,5 \lra \ell$, there are integers $i,j$ with $7\le i,j\le k$ such that $a_i=4$ and $a_j=5$. 
Since $\langle 3,3,3,4,5\rangle$ is tight $\mathcal T(3)$-universal by Proposition \ref{propg3}, $\ell$ is not new. This completes the proof. 
\end{proof}

\begin{table}[ht]
\caption{New diagonal tight $\mathcal T(2)$-universal $\z$-lattices $\langle a_1,a_2,\dots,a_k\rangle$} 
\label{tableg2}
	\begin{center}
		\begin{tabular}{|ccccc|c|}
			\Xhline{1pt}
			$a_1$&$a_2$&$a_3$&$a_4$&$a_5$& Conditions on $a_5$\\
			\hline
			2&2&2&2&3&\\
			2&2&2&3&$a_5$&$3\le a_5\le 17,\ a_5\neq 4,16$\\
			2&2&3&3&$a_5$&$3\le a_5\le 9,\ a_5\neq 4,8$\\
			2&2&3&4&&\\
			2&2&3&5&6&\\
			2&2&3&6&$a_5$&$6\le a_5\le 15,\ a_5\neq 14$\\
			\hline
			2&3&3&3&4&\\
			2&3&3&4&$a_5$&$4\le a_5\le 13,\ a_5\neq 5,8,12$\\
			2&3&4&4&$a_5$&$4\le a_5\le 17,\ a_5\neq 5,8,16$\\
			2&3&4&5&&\\
			2&3&4&6&$a_5$&$6\le a_5\le 23,\ a_5\neq 8,22$\\
			2&3&4&7&$a_5$&$7\le a_5\le 17,\ a_5\neq 8,16$\\
			2&3&4&8&&\\
			2&3&4&9&10&\\
			2&3&4&10&$a_5$&$10\le a_5\le 23,\ a_5\neq 22$\\
			\hline	
			\Xhline{1pt}
		\end{tabular}
	\end{center}
\end{table}

\begin{prop} \label{propg2}
Every diagonal $\z$-lattice listed in Table \ref{tableg2} is tight $\mathcal T(2)$-universal.
\end{prop}

\begin{proof} Let $L$ be any $\z$-lattice listed in Table \ref{tableg2}.  One may directly check that  $L$ represents all integers $m$ with $2\le m\le 575$.
Hence it suffices to show that any integer $u$ greater than $575$ is represented by $L$. 
\bigskip 

\noindent {\bf (i)} When $\langle 2,3,4,t\rangle \lra L$ for some positive odd integer $t$,
we put $K_1=\langle 2,3,4\rangle$.
Note that $h(K_1)=2$ and the other $\z$-lattice in the genus of $K_1$ is $K_2=\langle 1,2,12\rangle$. 
If an even positive integer is represented by $K_2$, then it is represented by $\langle 4,2,12\rangle$ and thus by $K_1$ too.
Therefore we have
$$
Q(\gen(K_1))\cap 2\z \subset Q(K_1).
$$
One may also  verify that there is an integer $d\in \{0,1,2,3,5\}$ such that
$$
u-td^2\in Q(\gen(K_1))\cap 2\z
$$
and thus we have $u-td^2\lra K_1$. Note that $u-td^2\ge u-23\cdot5^2>0$. 
Therefore we have $u\lra \langle 2,3,4,t\rangle \lra L$.
Since the proofs of all the remaining cases are quite similar to this, we only provide some data that are needed for each case. We first provide a ternary $\z$-lattice $K$ with $h(K)=1$ and a positive integer $b$ or positive integers $b_1$ and $b_2$ such that 
$$
K\perp \langle b\rangle \lra L \quad \text{or} \quad  K\perp \langle b_1,b_2\rangle \lra L.
$$ 
Finally, we provide a finite set $H$ in the former case or finite sets $H_1$ and $H_2$ in the latter case such that there is an integer $h \in H$ such that $u-bh^2 \lra K$ and hence $u \lra L$, or there are integers $h_1 \in H_1$ and $h_2 \in H_2$ such that $u-b_1h_1^2-b_2h_2^2 \lra K$.  Since we always choose a $\z$-lattice $K$ with class number $1$, the existence of a global representation is equivalent to the existence of a local representation. 

\noindent {\bf (ii)} When $\langle 2,2,3,t\rangle \lra L$ for some $t\in \n$ with $t\not\equiv 0\Mod 8$ and $t\not\equiv 0\Mod 3$,
we take $K=\langle 2,2,3\rangle$, $b=t$, and $H=\{0,1,3,4\}$.

\noindent {\bf (iii)} When  $L=\langle 2,3,4,8\rangle$, we 
take $K=\langle 2,4,8\rangle$, $b=3$, and $H=\{0,1,2,3,4\}$.

\noindent {\bf (iv)} When $L=\langle 2,2,3,a_4,a_5\rangle$ for some 
$$
(a_4,a_5)\in \{(3,3),(3,6),(3,9),(6,6),(6,9),(6,12),(6,15)\},
$$ 
we take $K=\langle 2,2,3\rangle$, $b_1=a_4$, $b_2=a_5$, and $H_1=\{0,1\}$, $H_2=\{0,1,2,3,4,5\}$.

\noindent {\bf (v)} When $L=\langle 2,3,4,4,a_5\rangle$ for some $a_5\in \{4,6,10,12,14\}$,
we take $K=\langle 2,4,4\rangle$, $b_1=3$, $b_2=a_5$, and $H_1=\{0,1,2,3\}$, $H_2=\{0,1,2\}$.

\noindent {\bf (vi)} When $L=\langle 2,3,4,6,a_5\rangle$ for some $a_5\in \{6,10,12,14,16,18,20\}$,
we take $K=\langle 2,4,6\rangle$, $b_1=3$, $b_2=a_5$, and $H_1=\{0,1,2,3\}$, $H_2=\{0,1,2\}$.

\noindent {\bf (vii)} When $L=\langle 2,3,4,10,a_5\rangle$ for some $a_5\in \{10,12,14,16,18,20\}$,
we take $K=\langle 2,4,10\rangle$, $b=3$, and $H=\{1,2,3,4\}$.

Since one of the above cases holds for any $\z$-lattice $L$ listed in Table \ref{tableg2}, this completes the proof.  
\end{proof}
\begin{prop} \label{propg2cand}
There are exactly $90$  new diagonal tight $\mathcal T(2)$-universal $\z$-lattices which are listed in Table \ref{tableg2}.
\end{prop}

\begin{proof}
Since the proof is quite similar to that of Proposition \ref{propg3cand} , we only provide some data that are needed  for the case when $n=2$. One may easily check by direct computations that 
$$
\begin{array}{llll}
\#A(3)=4,&\#B(3)=0,&\#B'(3)=0,&\#C(3)=4,\\[0.2em]
\#A(4)=15,&\#B(4)=3,&\#B'(4)=3,&\#C(4)=12,\\[0.2em]
\#A(5)=107,&\#B(5)=95,&\#B'(5)=87,&\#C(5)=12,\\[0.2em]
\end{array}
$$
and
$$
C(5)=\left\{ \begin{array}{llll} 
(2,2,2,2,2),&(2,2,2,3,16),&(2,2,3,3,8),&(2,2,3,5,5),\\[0.2em]
(2,2,3,6,14),&(2,3,3,3,3),&(2,3,3,4,12),&(2,3,4,4,16),\\[0.2em]
(2,3,4,6,22),&(2,3,4,7,16),&(2,3,4,9,9),&(2,3,4,10,22)\end{array}\right\}.
$$
For any $(a_1,a_2,\dots,a_5) \in C(5)$, one may directly show that there is no new diagonal tight $\mathcal T(2)$-universal $\z$-lattice $K$ such that $\langle a_1,a_2,\dots,a_5\rangle \prec K$ by using Lemma \ref{lemdiag2}.
This completes the proof.
\end{proof}

\section{ The minimum rank of tight $\mathcal T(n)$-universal lattices}

In this section, we give some explicit lower and upper bounds for $t(n)$ for any integer $n$ less than or equal to $36$. As explained in the introduction, we know that 
$$
t(1)=t(2)=t(3)=4.
$$ 
Furthermore, it is known that there are exactly $204$ tight $\mathcal T(1)$-universal quaternary $\z$-lattices (see \cite{b}). 


Let $n$ be a positive integer greater than $1$. 
To find all tight $\mathcal T(n)$-universal $\z$-lattices, one may use, so called, the escalation method (for this, see \cite{b}).  Let $L$ be a tight $\mathcal T(n)$-universal $\z$-lattice. Since $n, n+1 \lra L$, 
$$
\ell(1,a_1):=\begin{pmatrix} n&a_1\\a_1&n\end{pmatrix} \lra L \quad \text{or} \quad  \ell(2,a_2):=\begin{pmatrix} n&a_2\\a_2&n+1\end{pmatrix} \lra L
$$
for some integers $a_1, a_2$ with $0\le a_1\le \frac n2$, $0\le a_2\le\frac{n+1}2$. One may easily check that both binary $\z$-lattices $\ell(1,a_1)$ and $\ell(2,a_2)$ do not represent some integer $k$ with $n+1 \le k \le 2n$. Furthermore, since $L$ does not represent any positive integer less than $n$, there are integers $b_i$ and $c_i$ such that 
$$
\begin{pmatrix} n&a_1&b_1\\a_1&n&c_1\\b_1&c_1&k\end{pmatrix} \lra L \quad \text{or} \quad \begin{pmatrix} n&a_2&b_2\\a_2&n+1&c_2\\b_2&c_2&k\end{pmatrix} \lra L,
$$ 
where $-n \le b_1,c_1\le n$ and  $-\frac{2n+1}2\le b_2,c_2\le \frac {2n+1}2$.  If one of ternary $\z$-lattices given above represents an integer less than $n$, then we remove it.  Now, by continuing this process, we may find all candidates of tight $\mathcal T(n)$-universal $\z$-lattices of given rank. We provide some experimental results on some candidates of tight $\mathcal T(n)$-universal $\z$-lattices obtained from the escalation method:  

\noindent {\bf (1) } For any integer $n$ with $2\le n\le 7$, there are exactly $m_2$ isometry classes of quaternary $\z$-lattices with minimum $n$ which represent all integers from $n$ to $m_1$, and in fact, all of those  $\z$-lattices represent all integers from $n$ to $10^4$.
\begin{table} [ht]
\begin{tabular}{|c|c|c|c|c|c|c|c|c|}
\hline
$n$ & 2 & 3 & 4 & 5 & 6 & 7 \\
\hline
$m_1$ & 41 & 61 & 82 & 91 & 105 & 131 \\
\hline
$m_2$ & 308 & 1294 & 1222 & 1826 & 490 & 116 \\
\hline
\end{tabular}
\end{table}

\noindent {\bf (2) } There are exactly two candidates of tight $\mathcal T(8)$-universal quaternary $\z$-lattices, which are, in fact,  
$$
\begin{pmatrix}8&1&4&3\\ 1&9&1&4\\ 4&1&9&0\\ 3&4&0&9\end{pmatrix}\quad \text{and}\quad \begin{pmatrix}8&1&2&3\\ 1&9&4&4\\ 2&4&9&3\\ 3&4&3&9\end{pmatrix}.
$$
\noindent   {\bf (3) } There is only one candidate of tight $\mathcal T(9)$-universal quaternary $\z$-lattices, which is, in fact,  

$$
\begin{pmatrix}9&3&-3&1\\ 3&9&1&3\\ -3&1&10&3\\ 1&3&3&11\end{pmatrix}.
$$

\noindent  {\bf (4) }There is no tight $\mathcal T(n)$-universal quaternary $\z$-lattice for any $n$ with $10\le n\le 12$, and thus $t(n)$ is greater than $4$ for those integer $n$.

\begin{rmk} {\rm Let $L$ be the quaternary $\z$-lattice which is given as a candidate of a tight $\mathcal T(9)$-universal $\z$-lattice.  Then one may easily check that $dL=2^3\cdot 701$ and $L_p$ is isotropic over $\z_p$  for any prime $p$.  Hence, by Theorem 1.6 of [\cite{cas}, Chapter 11],  there is an integer $N$ such that $L$ represents all integers greater than $N$. Hanke developed in \cite{han} a method to compute the constant $N$ explicitly. However, the  constant $N$ in this case seems to be too large to check the representability of all integers less than $N$ by the $\z$-lattice $L$.}
\end{rmk}

\begin{lem} \label{lemnnnn}
Let $L$ be a $\z$-lattice and let $n$ be a positive integer.
If $\min(L)\in \{n,n+1\}$ and $L$ represents all integers from $n+1$ to $2n-1$, then the $\z$-lattice $ \langle n,n,n,n\rangle \perp L$ is tight $\mathcal T(n)$-universal, and thus $t(n)\le \rank (L)+4$.
\end{lem}

\begin{proof}
This is quite straightforward.
\end{proof}

\begin{lem} \label{lemnn2n}
Let $L$ be a $\z$-lattice and let $n$ be a positive integer.
Assume that $\min(L)\in \{n,n+1\}$ and $L$ represents all integers from $n+1$ to $2n-1$.
Assume further that there is an integer $s$ with $n\le s\le 14n$ such that $L$ represents all integers from $s$ to $s+2n-1$.
Then the $\z$-lattice $\langle n,n,2n\rangle \perp L$ is tight $\mathcal T(n)$-universal, and thus $t(n)\le \rank (L)+3$.
\end{lem}

\begin{proof}
Let $K=\langle n,n,2n\rangle \perp L$.
From the fact that $\langle 1,1,2\rangle$ represents all integers from $0$ to $13$, and the assumption that $L$ represents all integers from $n+1$ to $2n-1$, one may easily deduce that all integers from $n$ to $14n-1$ are represented by $K$. Let $m$ be any integer greater than or equal to $s$. Let $a$ be any integer with $s\le a \le s+n-1$ such that $m\equiv a \Mod n$. Then there is a nonnegative integer $u$ such that $m=nu+a$. If $u=0$, then clearly, $m=a \lra K$. Now, assume that $u\ge 1$. Since either $u$ or $u-1$ is odd and any odd integer is represented by $\langle 1,1,2\rangle$, either $nu$ or $n(u-1)$ is represented by $\langle n,n,2n\rangle$. Furthermore, since both $a$ and $a+n$ are represented by $L$ from the assumption,    the integer
$$
m=nu+a=n(u-1)+a+n
$$
is represented by $K$.  Now, the lemma follows from the assumption that $s\le 14n$.
\end{proof}
\begin{prop}
For any integer $n$ with $4\le n\le 14$, we have
$$
t(n) \le \begin{cases} n+1\quad &\text{if $4\le n\le 6$},\\
7\quad   &\text{if $7\le n\le 14$}.
\end{cases}
$$
\end{prop}

\begin{proof}
Note that $t(n) \le n+1$ for any $n \ge 4$ by Proposition \ref{propdiag1}.  We define 
$$
\begin{array}{ll}
L(7)=\begin{pmatrix}7&0&3&3\\0&7&3&1\\3&3&7&1\\3&1&1&7\end{pmatrix},&
L(8)=\begin{pmatrix}8&1&4&3\\1&9&1&4\\4&1&9&0\\3&4&0&9\end{pmatrix},\\[2.5em]
L(9)=\begin{pmatrix}9&3&-3&1\\3&9&1&3\\-3&1&10&3\\1&3&3&11\end{pmatrix},&
L(10)=\begin{pmatrix}10&0&1&3\\0&10&3&4\\1&3&11&-3\\3&4&-3&11\end{pmatrix},\\[2.5em]

\end{array}
$$
$$
\begin{array}{ll}
L(11)=\begin{pmatrix}11&-1&1&3\\-1&11&3&4\\1&3&12&-4\\3&4&-4&12\end{pmatrix},&
L(12)=\begin{pmatrix}12&1&4&0\\1&13&-3&3\\4&-3&13&6\\0&3&6&14\end{pmatrix},\\[2.5em]
L(13)=\begin{pmatrix}13&0&3&4\\0&14&5&5\\3&5&14&-3\\4&5&-3&15\end{pmatrix},&
L(14)=\begin{pmatrix}14&2&2&6\\2&15&2&7\\2&2&16&-5\\6&7&-5&17\end{pmatrix}.\\[2.5em]
\end{array}
$$
If we take $s=7,8,9,98,124,112,80,156$, when $n=7,8,\cdots,14$, respectively, then one may easily check by direct computations that for each $n$, the corresponding $\z$-lattice $L(n)$ and the integer $s$ satisfy all conditions given in  Lemma \ref{lemnn2n}.  This completes the proof. \end{proof}
For a positive integer $u$ and an integer $v$, we define a quinary $\z$-lattice 
$$
K(u,v)=\begin{pmatrix}2u+v&u&u&u&u-4\\u&2u+1&0&0&0\\u&0&2u+2&0&0\\u&0&0&2u+4&0\\u-4&0&0&0&2u\end{pmatrix}.
$$
\begin{prop}
For any integer $n$ with  $15\le n\le 21$, we have $t(n)\le 8$.
\end{prop}

\begin{proof}
One may directly check that for each integer $n$, the corresponding quinary $\z$-lattice $K(u,v)$ and the integer $s$ satisfy all conditions given in Lemma \ref{lemnn2n}  in the following table:

\begin{table} [ht]
\begin{tabular}{|c|c|c|c|c|}
\hline
$n$ & 15,16 & 17,18 & 19,20 & 21\\
\hline
$(u,v)$ & $(8,0)$ & $(9,2)$ & $(10,3)$ & $(11,3)$\\
\hline
$s$ & 49 & 57 & 96 & 106\\
\hline
\end{tabular}
\end{table}

Hence the proposition follows immediately from Lemma \ref{lemnn2n}.
\end{proof}
For a positive integer $u$ and an integer $v$, we define a senary $\z$-lattice $L(u,v)$ as
$$
L(u,v)=\begin{pmatrix}2u+v&u&u&u-2&u-4&u-8\\u&2u+1&0&0&0&0\\u&0&2u+2&0&0&0\\u-2&0&0&2u&0&0\\u-4&0&0&0&2u&0\\u-8&0&0&0&0&2u\end{pmatrix}.
$$
\begin{prop}
For any integer $n$ with $22\le n\le 36$, we have $t(n)\le 9$.
\end{prop}

\begin{proof}
One may directly check that for each integer $n$, the corresponding senary $\z$-lattice $L(u,v)$ and the integer $s$ satisfy all conditions given in Lemma \ref{lemnn2n} in the following table:

\newpage
\begin{table} [ht]
\begin{tabular}{|c|c|c|c|c|c|c|c|c|}
\hline
$n$ & 22 & 23,24 & 25,26 & 27,28 & 29,30 & 31,32 & 33,34 & 35,36 \\
\hline
$(u,v)$ & $(11,0)$ & $(12,0)$ & $(13,0)$ & $(14,1)$ & $(15,1)$ & $(16,2)$ & $(17,3)$ & $(18,4)$\\
\hline 
$s$ & 22 & 24 & 61 & 65 & 71 & 81 & 86 & 93\\
\hline
\end{tabular}
\end{table}

Hence the proposition follows immediately from Lemma \ref{lemnn2n}.
\end{proof}

\section{Lower and upper bounds for $t(n)$}
In this section, we find some explicit lower and upper bounds for the minimum rank $t(n)$ of tight $\mathcal T(n)$-universal $\z$-lattices.  
 
\begin{thm} \label{thmlowerB}  For any positive integer $n$, we have 
$$
\max(4,\log_2(n+4)) \le t(n).
$$ 
In particular, the minimum rank of tight $\mathcal T(n)$-universal $\z$-lattices goes to infinity as $n$ increases. 
\end{thm}

\begin{proof} Let $L$ be a tight $\mathcal T(n)$-universal $\z$-lattice of rank $k$.  Since $L$ represents all integers over $\z_p$ for any prime $p$,  $k$ is greater than or equal to $4$.  Recall that any ternary $\z$-lattice $\ell$ has always an anisotropic prime, say $q$, and $\ell$ does not represent some integer over $\z_q$. 

For each coset $\gamma \in (L/2L)^{\times}$,  the minimum of  $\gamma$ is defined by
$$
\min(\gamma)=\min\{ Q(\bx) : x \in \gamma\}.
$$
Let $\alpha \in (L/2L)^{\times}$ be a coset  and let $\bx,\by \in \alpha$ be vectors such that $\bx \ne \pm \by$. Since $\bx\pm \by \in (2L)^{\times}$,  we have
$$
Q(\bx\pm \by)=Q(\bx)\pm 2B(\bx,\by)+Q(\by) \ge 4n.
$$ 
Therefore we have $Q(\bx)+Q(\by) \ge 4n$. For each positive integer $m$ greater than or equal to $n$, we fix a vector $\bx_m \in L$  such that $Q(\bx_m)=m$. Let $\alpha_m$ be the coset in $(L/2L)^{\times}$ containing the vector $\bx_m$.     
Then for any different integers $r$ and  $s$ with $n\le r,s\le 2n$, one may easily check that $\alpha_r \ne \alpha_s$ from the above observation. 

Now, suppose that there is an integer $u$ with $n\le u\le 2n$ such that $\bx_{2n+1} \in \alpha_u$.  Then since there is a vector $\by \in L$ such that $\bx_{2n+1}=\bx_u+2\by$, we have $2n+1 \equiv u \Mod 4$. Furthermore, since $2n+1+u \ge 4n$, such an integer $u$ does not exist.  Suppose that $\bx_{2n+2} \in \bigcup_{u=n}^{2n+1} \alpha_u$.  Then by using a similar argument given above, one may easily show that the only coset which could possibly contain $\bx_{2n+2}$ is $\alpha_{2n-2}$.  
Since $\bx_{2n+2} \pm \bx_{2n-2} \in (2L)^{\times}$, we have
$$
Q(\bx_{2n+2}+\bx_{2n-2})=2n+2+2n-2\pm2B(\bx_{2n+2},\bx_{2n-2}) \ge 4n.
$$
This implies that $B(\bx_{2n+2},\bx_{2n-2})=0$. Therefore we have 
$$
Q\left(\frac {\bx_{2n+2}+\bx_{2n-2}}2\right)=Q\left(\frac {\bx_{2n+2}-\bx_{2n-2}}2\right)=n.
$$
Consequently,  there are at least two different cosets whose minimum is $n$.  Summing up all, there are at least $n+3$ different cosets in $(L/2L)^{\times}$, and hence $n+3 \le 2^k-1$. This completes the proof.
\end{proof}

\begin{rmk} \label{lower5}   {\rm Note that $t(n) \ge 5$ for any $n \ge 13$ from the above theorem. As mentioned in Section 3, we have checked by using the escalation method that  there is no quaternary tight $\mathcal T(n)$-universal $\z$-lattice for any integer $n$ with $10\le n\le 12$. Therefore $t(n)\ge 5$ for any integer $n$ greater than or equal to $10$.}  
\end{rmk}

As we have shown in Proposition \ref{propdiag1}, $t(n) \le n+1$. Now, we find some explicit upper bound for $t(n)$, which is better than this.
\begin{thm} For any positive integer $n$, we have
$$
t(n)\le \left[ \frac{\left[ \sqrt n\right]-1}2\right]+\left[\sqrt n\right]+7.
$$
\end{thm}

\begin{proof}   Since $t(n)=4$ for any integer $n$ with $1\le n\le 3$, we always assume that $n\ge 4$.
For simplicity, we define
$$
s_n=\left[ \frac{\left[ \sqrt n\right] -1}2\right].
$$
Put $m=s_n+\left[ \sqrt n\right]+3$ and let $M_1=\z \bx_1+\z \bx_2+\dots+\z \bx_{m}$ be the $\z$-lattice of rank $m$ such that
$$
B(\bx_i,\bx_j)=\begin{cases} 2\left[\frac n2\right] \quad &\text{if $i=j$},\\
-\left[\frac n2\right] \quad &\text{if $\vert i-j\vert=1$ or $m-1$},\\
0\quad &\text{otherwise}. \end{cases}
$$
Note that $M_1$ is a positive semi-definite even $\z$-lattice with  
$$
Q(\bx_1+\bx_2+\dots+\bx_m)=0,
$$
and
$$
Q(\bx) \ge 2\left[\frac n2\right]  \ \ \text{for any $\bx \in M_1-\z(\bx_1+\bx_2+\dots+\bx_m)$}.
$$
In fact, one may easily check that 
$$
\z \bx_i+\z \bx_{i+1}+\dots+\z \bx_{i+m-2} \simeq A_{m-1}^{\left[\frac n2\right]}
$$
for any $i$ with $1\le i \le m$. Here we are assuming that $\bx_{\nu}=\bx_{\nu-m}$ when $\nu >m$.   
Note that $A_{m-1}^{\left[\frac n2\right]}$ is the $\z$-lattice of rank $m-1$ obtained from the root lattice $A_{m-1}$ by scaling by $\left[\frac n2\right]$.                                    
Let $M_2=\z \by_1+\z \by_2+\dots+\z \by_m$ be the diagonal $\z$-lattice of rank $m$ such that
$$
Q(\by_i)=\begin{cases} 1 \quad &\text{if $i=1$}, \\
                    2 \quad &\text{if $2\le i\le s_n+1$},  \\
                    \left[ \sqrt n\right] \quad &\text{otherwise.}  \end{cases}
$$
          
Finally, let $L=\z(\bx_1+\by_1)+\z(\bx_2+\by_2)+\dots+\z(\bx_m+\by_m) \subset M_1\perp M_2$ be the $\z$-lattice of rank $m$. From the definition, one may easily show that 
$$
\min(L)=\min\left(2\left[\frac n2\right]+1, 1+2s_n+\left[\sqrt n\right]\cdot \left( \left[\sqrt n\right] +2\right) \right)=2\left[\frac n2\right]+1\ge n.
$$  

Now,  we show that $L$ represents all integers $k$ with $n+1\le k\le 2n-1$.
Let $u$ be an integer with $1\le u\le n$ such that $k=2\left[ \frac n2 \right]+u$.
One may easily show that there is an integer $q_0$ with $0\le q_0\le \left[ \sqrt n\right] +2$ and an integer $t_0$ with $0\le t_0\le s_n$ such that either $u=\left[ \sqrt n\right] \cdot q_0+2t_0$ or $u=\left[ \sqrt n\right] \cdot q_0+2t_0+1$.
For the former case, one may easily check that $2\le s_n-t_0+2\le s_n+q_0+1\le m$ and thus we have
$$
2\left[ \frac n2\right] +u=2\left[ \frac n2\right] +\left[ \sqrt n\right] \cdot q_0+2t_0=Q\left(\sum_{u=s_n-t_0+2}^{s_n+q_0+1}(\bx_u+\by_u)\right).
$$
For the latter case, 
since
$$
u=\left[ \sqrt n\right] \cdot q_0+2t_0+1 \le n<\left( \left[ \sqrt n\right] +1\right)^2\le \left[ \sqrt n\right] \cdot \left( \left[ \sqrt n\right] +2\right) +2t_0+1,
$$
we have  $q_0\le \left[ \sqrt n\right]+1$.  This implies that
$$
3+s_n\le m+1-q_0\le m+1+t_0\le m+1+s_n.
$$
Hence we have
$$
2\left[ \frac n2\right] +u=2\left[ \frac n2\right] +\left[ \sqrt n\right] \cdot q_0+2t_0+1=Q\left(\sum_{u=m+1-q_0}^{m+1+t_0}(\bx_u+\by_u)\right),
$$
where we are assuming that $\bx_{\nu}=\bx_{\nu-m}$ and $\by_{\nu}=\by_{\nu-m}$ when $m+1\le \nu \le m+1+t_0$.   

Now, since the $\z$-lattice $L$ represents all integers $k$ with $n+1\le k\le 2n-1$, the $\z$-lattice $\langle n,n,n,n\rangle \perp L$ is tight $\mathcal T(n)$-universal by Lemma \ref{lemnnnn}. Therefore we have
$$
t(n) \le \rank(L)+4\le s_n+\left[ \sqrt n\right] +7.
$$
This completes the proof.
\end{proof}



\begin{thebibliography}{abcd}

\bibitem{bar} M. Barowsky, W. Damron,  A. Mejia,  F. Saia,  N. Schock and K. Thompson,  {\em Classically integral quadratic forms excepting at most two values}, Proc. Amer. Math. Soc. \textbf{146}(2018), 3661-3677.

\bibitem {b} M. Bhargava, {\em On the Conway-Schneeberger fifteen theorem}, Quadratic forms and their applications (Dublin, 1999), 27-37, Contemp. Math., \textbf{272}, Amer. Math. Soc., Providence, RI, 2000.

\bibitem {cas} J. W. S. Cassels, {\em Rational quadratic forms},  London Mathematical Society Monographs, 13. Academic Press, Inc., London-New York, 1978.

\bibitem {dic} L. E. Dickson, {\em History of the Theory of Numbers, Vol. III: Quadratic and Higher Forms}, reprint, Chelsea, New York, 1966. 

\bibitem {hal} P. R. Halmos, {\em Note on almost-universal forms}, Bull. Amer. Math. Soc. \textbf{44}(1938), 141-144.

\bibitem {han}  J.  Hanke, {\em Local densities and explicit bounds for representability by a quadratric form}, Duke Math. J. \textbf{124}(2004),  351-388.

\bibitem{hi}  J. S. Hsia and M. I. Icaza,  {\em Effective version of Tartakowsky's theorem}, Acta Arith. \textbf{89}(1999),  235-253.

\bibitem {hs} J. S. Hsia, {\em Representations by spinor genera}, Pacific J. Math. \textbf{63}(1976), 147-152.

\bibitem {ki} Y. Kitaoka, {\em Arithmetic of quadratic forms}, Cambridge University Press, 1993.

\bibitem{om} O. T. O'Meara, {\em Introduction to quadratic forms}, Springer-Verlag, New York, 1963.

\bibitem {om2} O. T. O'Meara,  {\em The integral representations of quadratic forms over local fields}, Amer. J. Math. \textbf{80}(1958), 843-878.

\bibitem{r} S. Ramanujan, {\em On the expression of a number in the form $ax^2 + by^2 + cz^2 + du^2$} , Proc. Cambridge Phil. Soc. \textbf{19}(1917), 11-21.

\end{thebibliography}
\end{document}